\documentclass[12pt,a4paper]{amsart}
 \usepackage{amssymb,amsthm}
  \usepackage{multirow}
 \newtheorem{thm}{Theorem}[section]
\newtheorem{cor}[thm]{Corollary}
\newtheorem{prop}[thm]{Proposition}
\newtheorem{lem}[thm]{Lemma}

\newtheorem{rem}[thm]{Remark}

\theoremstyle{definition}
\newtheorem{defn}[thm]{Definition}

\begin{document}
 \title[A characterization of always solvable trees]
 {A characterization of always solvable trees in the Lights Out game using the activation types of vertices}
 \author{Ahmet Batal}
 \address{Department of Mathematics\\ Izmir Institute of Technology\\Izmir, TURKEY}
\email{ahmetbatal@iyte.edu.tr}
 \date{}
\begin{abstract}
Lights out is a game that can be played on any simple graph $G$. A configuration assigns one of the two states \emph{on} or \emph{off} to each vertex. For a given configuration, the aim of the game is to turn all vertices \emph{off} by applying a push pattern on vertices, where each push switches the state of the vertex and its neighbors. If every configuration of vertices is solvable, then we say that the graph is always solvable. We introduce a concept which we call the activation types of vertices and we prove several characterization results of trees by using this concept. We showed that all always solvable trees different than the star tree can be seen as the join graph of its two always solvable subtrees. We call the dimension of the space of null-patterns, which leave configurations unchanged, the nullity of the graph $G$. We show that the nullity of a tree can be characterized by the cardinality of its minimal partition into always solvable subtrees. We also showed that nullity of a tree is less than the number of its even degree vertices.
\end{abstract}

 \maketitle

 \section{Introduction}

Lights out game on an undirected graph $G(V,E)$ is played as follows. Each vertex $v$ has a state which is either \emph{on} or \emph{off}. When a vertex is pushed, the vertex itself and all of its neighbors switch state. This push is called \emph{the activation of vertex} $v$. For any given initial \emph{on/off} configuration of the vertices, the aim of the game is to turn all the vertices \emph{off} by a sequence of activations. It is easy to observe that the order of the activations has no importance nor the act of activating a vertex more than once. Hence, the sequence of activations can be identified by the set $P$ of activated vertices, which we call an \emph{activation pattern} (or simply a \emph{pattern}). Similarly, each initial configuration can be identified by a set $C$ such that $v\in C$ if the state of $v$ is \emph{on} in the configuration. If there exists an activation pattern $P$ which turns all lights \emph{off} for a given initial configuration $C$, then the configuration $C$ is called \emph{solvable} and $P$ is called \emph{a solving pattern for} $C$ (we may also say $P$ \emph{solves} $C$).
If every configuration is solvable, then the graph $G$ is called \emph{always solvable} (it is also called \emph{all parity realizable}, \emph{always winnable}, \emph{universally solvable} by other authors \cite{Amin96}, \cite{Giffen13}, \cite{Fleischer13}).

Let the order of $V$ be $n$ and let $\{v_1,...,v_n\}$ be an enumeration of $V$. Then any subset $S\subseteq V$ can be represented by its \emph{characteristic column vector} $\mathbf{s}=(s_1,...,s_n)^t$ where $\mathbf{s}(v_i):=s_i=1$ if $v_i\in S$, and $s_i=0$ otherwise. Hence, patterns or configurations can be identified by characteristic vectors of sets as well. $N[v]=\{u\in V \;|\; u=v \;\text{or}\; u \;\text{is adjacent to}\; v\}$ is called \emph{the closed neighborhood set of vertex} $v$, and $n\times n$ matrix $N:=N(G)$ whose $i$th column is the characteristic vector of $N[v_i]$ is called \emph{the closed adjacency matrix of} $G$. We denote Kernel, column space, and row space of $N$  by $Ker(N)$, $Col(N)$, and $Row(N)$, respectively. Let $\nu(G):=dim(Ker(N(G))$ and $r(G):=dim(Col(N(G))$. We call $\nu(G)$ \emph{nullity of} $G$ (Amin et al. \cite{Amin98} call it parity dimension of $G$). By rank nullity theorem, we have $\nu(G)+r(G)=n$. It was first observed by Sutner \cite{Sutner89},\cite{Sutner90} that an activation pattern $\mathbf{p}$ is a solving pattern for an initial configuration $\mathbf{c}$ \emph{iff}
\begin{equation}\label{Npc}
N\mathbf{p}=\mathbf{c}
\end{equation}
 over the field $\mathbb{Z}_2$. This simple linear algebraic formulation of the game leads to several realizations \cite{Sutner89},\cite{Sutner90},\cite{Amin98}:

(\emph{R 1}) A configuration $\mathbf{c}$ is solvable \emph{iff} $\mathbf{c} \in Col(N) $. Hence, graph $G$ is always solvable  \emph{iff} $r(G)=n$, which is equivalent to say that $\nu(G)=0 $.

(\emph{R 2}) Number of solving patterns of a given configuration is $2^{\nu(G)}$. Indeed, if $\mathbf{p}$ solves $\mathbf{c}$, then $\mathbf{p}+\boldsymbol{\ell}$ solves $\mathbf{c}$ as well for every $\boldsymbol{\ell}\in Ker(N)$. Members of $Ker(N)$ are called \emph{null patterns} since they have no effect on a given configuration.

(\emph{R 3}) Since $N$ is a symmetric matrix, we have $Col(N)=Row(N)$. Moreover, $Row(N)=Ker(N)^\bot$, where $S^\bot$ denotes the orthogonal complement of a set $S$ with respect to \emph{the dot product} $\mathbf{x}\cdot \mathbf{y} := \mathbf{x}^t \mathbf{y}$. Hence, a configuration is solvable  \emph{iff} it is orthogonal to every null pattern.


For a given graph, finding the pattern $\mathbf{p}$ which solves all lights on configuration $\mathbf{c}=\mathbf{1}$ is called \emph{all-ones problem} and the configuration $\mathbf{1}$ is called \emph{all-ones configuration}. The reason why it deserves a special name is the fact that all-ones problem is solvable for all graphs \cite{Sutner89} (see also \cite{Caro96}, \cite{Cowen99}, \cite{Erikson04}). In other words, for every graph $G$ there exists a pattern $\mathbf{p}$ such that $N(G)\mathbf{p}=\mathbf{1}$.


The connection between the nullities of graphs related to each other by some type of graph operation such as edge/vertex join or removal was first considered by Amin et al. \cite{Amin02}. The same subject was investigated by  Giffen et al. \cite{Giffen13}, Edwards et al. \cite{Edwards10} and Ballard et al. \cite{Ballard19} for a generalized version of the game where \eqref{Npc} is considered over $\mathbb{Z}_k$ for some integer $k\geq 2.$ Let us denote the graph obtained by removing a vertex $u$ and all its incident edges from a graph $G$ by $G-u$. In \cite{Amin02} and \cite{Ballard19}, the difference $\nu(G-u)-\nu(G)$ plays an important role in their analysis. We adapt the terminology of \cite{Ballard19} and call this difference \emph{the null difference of vertex} $u$ and denote it by $nd(u)$. It turns out that $nd(u)$ can be either $-1$, $0$, or $1$; and as a result, vertices can be categorized by their null difference number into three different classes \cite{Amin02}, \cite{Ballard19}.

 We show that these classes can be realized from a different point of view in the case of classical lights out game where $k=2.$ Indeed, there is another way of categorizing vertices by checking in how many solving patterns of all-ones configuration a vertex is activated. This categorization again leads to three different classes. We prove that the first set of classes determined by the null differences of vertices and the second one determined by the activation of the vertices under solving patterns of all-ones configuration, coincide with each other.

Seeing this set of classes from the second point of view has important advantages. For example, this point of view allows us to track how the class of some vertices changes under some special graph operations (see Theorem \ref{main}), and enables us to make three observations about trees. Our first observation is that the nullity of a tree can be characterized by the cardinality of its minimal partition into always solvable subtrees as we showed in Theorem \ref{partition}. Our second observation is related to the characterization of always solvable trees. In Theorem \ref{chr2} we showed that an always solvable tree can be characterized as some special join tree of its two always solvable subtrees if it is different than the star tree. This gives an alternative to the previous two characterizations of always solvable trees given in \cite{Amin96} and \cite{Amin02}. We also showed in Theorem \ref{thmN} that the nullity of a tree with at least one even degree vertex is less than the number of its even degree vertices.


\section{Activation types of vertices}

\begin{defn}
We define the inverse configuration $\overline{\mathbf{c}}$ of a configuration $\mathbf{c}$ as $\overline{\mathbf{c}}=\mathbf{c}+\mathbf{1}$. For a given vertex $u$, we denote the configuration, where only the state of $u$ is on by $\mathbf{c}_u$ i.e.; $\mathbf{c}_u(v)=1$ iff $v=u$.
\end{defn}

Let $u$ and $w$ be two vertices of a graph $G$ which are adjacent to each other. Let us denote the graph obtained by deleting the edge $e=(u,w)$ between $u$ and $w$ by $G-e$. Let $V(G)=\{v_1,...v_n\}$. Then the difference between the neighborhood matrices $N(G-e)$ and $N(G)$ is the symmetric matrix $J$ where $J_{ik}=J_{ki}=1$ if $v_i=u,\; v_k=w$, and $0$ otherwise. Let $\mathbf{p}$ be a solving pattern for a configuration $\mathbf{c}$ on $G$. Then the same pattern applied to the graph $G-e$ solves the configuration
\begin{equation}
\label{edgedel}
\tilde{\mathbf{c}}=N(G-e)\mathbf{p}=N(G)\mathbf{p}+J\mathbf{p}=\mathbf{c}+\mathbf{p}(w)\mathbf{c}_u+\mathbf{p}(u)\mathbf{c}_w.
\end{equation}
In other words, for a given pattern $\mathbf{p}$, deletion of an edge $e=(u,w)$ may only change the states of vertices $u$ and $w$ in the corresponding configuration, and the state of $u$ ($w$) changes if and only if $w$ ($u$) is activated in $\mathbf{p}$.

\begin{lem}
\label{NSS}
Let $u$ be a vertex such that the configuration $\overline{\mathbf{c}_u}$ is solvable. Then $u$ is not activated in any solving pattern $\mathbf{p}$ for $\overline{\mathbf{c}_u}$. In other words $\mathbf{p}(u)=0$ for all $\mathbf{p}$ satisfying $N\mathbf{p}=\overline{\mathbf{c}_u}$.
\end{lem}
\begin{proof}
We prove the lemma by applying induction on the size of the graph. Note that the lemma  holds true for any graph without any edges. Assume that it holds true for any graph with size $n$. Let $G$ be a graph with size $n+1$. Assume for a contradiction that there exists a pattern $\mathbf{p}$  such that $\mathbf{p}(u)=1$ and $N(G)\mathbf{p}=\overline{\mathbf{c}_u}$ for some vertex $u\in V(G)$. Since $\overline{\mathbf{c}_u}(u)=0$ and  $\mathbf{p}(u)=1$, there must exist odd number of activated vertices adjacent to $u$. Let $w$ be one of those adjacent vertices to $u$ such that $\mathbf{p}(w)=1$. Let $e=(u,w)$. Then by \eqref{edgedel} the pattern $\mathbf{p}$ applied to the graph $G-e$ solves the configuration $\overline{\mathbf{\mathbf{c}}_u}+\mathbf{c}_u + \mathbf{c}_w=\mathbf{1}+\mathbf{c}_w=\overline{\mathbf{\mathbf{c}}_w}$. Since $G-e$ has $n$ edges and $\mathbf{p}(w)=1$, this contradicts with the induction hypothesis.

\end{proof}
\begin{defn}
We call a vertex $v$ half-activated if $\boldsymbol{\ell}(v)=1$ for some null-pattern $\boldsymbol{\ell}$,  and fixed otherwise.
\end{defn}
By [\cite{Amin98}, Lemma 1] we know that  for any vertex $v\in V(G)$, either $\boldsymbol{\ell}(v)=0$ for all null-patterns $\boldsymbol{\ell}$ or $\boldsymbol{\ell}(v)=1$ for exactly half of the null-patterns which correspond to $2^{\nu(G)-1}$ patterns. Together with (\emph{R 2}), this gives us the following lemma.
\begin{lem}
Let $\mathbf{c}$ be a solvable configuration on a graph $G$. $v$ is half-activated if and only if $\mathbf{p}(v)=1$ for exactly half of the solving patterns $\mathbf{p}$ for $\mathbf{c}$. $v$ is fixed if and only if either

Case i) $\mathbf{p}(v)=1$  for all solving patterns $\mathbf{p}$ for $\mathbf{c}$, or

Case ii) $\mathbf{p}(v)=0$ for all solving patterns $\mathbf{p}$ for $\mathbf{c}$.

\end{lem}
Above lemma motivates us to make the following definition.
\begin{defn}
For a given solvable configuration $\mathbf{c}$, we call a fixed vertex $v$, $\mathbf{c}$-always-activated or $\mathbf{c}$-never-activated if $v$ satisfies Case (i) or Case (ii), respectively. In the case of $\mathbf{c}=\mathbf{1}$, instead of calling a vertex $\mathbf{1}$-always-activated ($\mathbf{1}$-never-activated), we simply call it always-activated (never-activated). When we say activation type of a vertex we mean its always, never, or half activatedness.
\end{defn}

\begin{lem}
\label{act}
Every always solvable graph has an always-activated vertex.
\end{lem}
\begin{proof}
Let $G$ be an always solvable graph. Since $\nu(G)=0$, there is a unique solution $\mathbf{p}$ for the all-ones configuration $\mathbf{1}$ on $G$. Because $\mathbf{p}\neq \mathbf{0}$ there must be a vertex $w\in V(G)$ with $\mathbf{p}(w)\neq 0$, i.e.; $w$ is activated. Since $\mathbf{p}$ is the only solution, $w$ is always-activated.
\end{proof}

As a consequence of \emph{(R 3)}, a vertex $v$ is fixed if and only if $\mathbf{c}_v$ is solvable; see [\cite{Ballard19}, Proposition 2.4]. Moreover, we have the following proposition.
\begin{prop}
A vertex $v$ is always-activated if and only if  $v$ is $\mathbf{c}_v$-always-activated.
\end{prop}
\begin{proof}
Let $\mathbf{s}$ and $\mathbf{p}$ be arbitrary solving patterns for $\mathbf{1}$ and $\mathbf{c}_v$, respectively. Then $\mathbf{s}+\mathbf{p}$ solves $\overline{\mathbf{c}_v}$. By Lemma \ref{NSS} $(\mathbf{s}+\mathbf{p})(v)=0.$ Thus, $\mathbf{s}(v)=\mathbf{p}(v)$.
\end{proof}

On the other hand, by [\cite{Ballard19}, Proposition 2.8, Proposition 2.10] we have $nd(v)=0$ ($nd(v)=1$) if and only if $v$ is $\mathbf{c}_v$-always-activated ($\mathbf{c}_v$-never-activated).
Thus, we reach the following identification, which connects the null difference number of a vertex to its activation type in the case of classical lights out game.

\begin{prop}
\label{actnd}
A vertex $v$ is always-activated if and only if $nd(v)=0$ and never-activated if and only if $nd(v)=1$.
\end{prop}

\section{Join of Graphs}

 Let $G_1$ and $G_2$ be two nonempty disjoint graphs and $H:=G_1uwG_2$ be the join graph constructed by joining the vertices $u$ of $G_1$ and $w$ of $G_2$ by an edge. Let $|G_i|=n_i$ where $i\in  \{1,2\}$. Then we can enumerate the vertices of $H$ as $\{v_1,...,v_{n_1+n_2}\}$ such that $V(G_1)=\{v_1,...,v_{n_1}\}$ with $v_{n_1}=u$ and $V(G_2)=\{v_{n_1+1},...,v_{n_1+n_2}\}$ with $v_{n_1+1}=w$. This way, we can represent every pattern $\mathbf{p}$ of $H$ as $\mathbf{p}^t=(\mathbf{p_1}^t,\mathbf{p_2}^t)$ where $\mathbf{p_i}$  is the restriction of $\mathbf{p}$ on $V(G_i)$.

 Let $\mathbf{s}$ be a solving pattern for the all-ones configuration on $H$, i.e.; $N(H)\mathbf{s}=\mathbf{1}$. Note that by \eqref{edgedel} we have

\begin{align}
\label{Ns1}
N(G_1)\mathbf{s_1}= \left\{ \begin{array}{cc}
                \mathbf{1} &  \text{if}\;\;\mathbf{s}(w)=0 \\
                \overline{\mathbf{c}_u} &  \text{if}\;\; \mathbf{s}(w)=1 \\
                \end{array} \right\},
\end{align}
and similarly
\begin{align}
N(G_2)\mathbf{s_2}= \left\{ \begin{array}{cc}
                \mathbf{1} &  \text{if}\;\;\mathbf{s}(u)=0 \\
              \overline{\mathbf{c}_w} &  \text{if}\;\; \mathbf{s}(u)=1 \\
                \end{array} \right\},
\end{align}
where we use the same notation $\mathbf{1}$ to denote the all-ones configurations of different graphs.

We already know that in any solving pattern for the all-ones configuration on a tree, adjacent vertices cannot be both activated \cite{Conlon99}. Moreover, we have the following lemma.
\begin{lem}
\label{cutedge}
Let $H$ be a graph with a cut edge $e=(u,w)$. Then in any solving pattern $\mathbf{s}$ for the all-ones configuration on $H$, vertices $u$ and $w$ cannot be both activated i.e.; either $\mathbf{s}(u)=0$ or $\mathbf{s}(w)=0$.
\end{lem}
\begin{proof}
Without loss of generality we can assume $H$ is connected. Further, we can see $H$ as $H=G_1uwG_2$ where $G_1$ ($G_2$) is the connected component containing $u$ ($w$) in $G-e$. Assume for a contradiction that there exists a solving pattern $\mathbf{s}$ for the all-ones configuration on $H$ such that $\mathbf{s}(u)=\mathbf{s}(w)=1$. Then by \eqref{Ns1} $N(G_1)\mathbf{s_1}=\overline{\mathbf{c}_u}$ with $\mathbf{s_1}(u)=\mathbf{s}(u)=1$, which contradicts with Lemma \ref{NSS}.
\end{proof}

\begin{defn}

Let $G_1,...,G_n$ be disjoint graphs. We define the join graph $H=G_1u_1u_2G_2u_2...u_nG_n$ as the graph obtained by joining the the vertices $u_i$ of $G_i$ and $u_{i+1}$ of $G_{i+1}$ by an edge for each $i\in\{1,...,n-1\}$.
Moreover, we say a join graph $H=G_1u_1u_2G_2u_2...u_nG_n$ is a Type-$(a_1,...,a_n)$ connection of disjoint graphs $G_1,...,G_n$, respectively, where $a_i=\mathcal{A}$ if $u_i$ is always-activated in $G_i$, $a_i=\mathcal{N}$ if $u_i$ is never-activated in $G_i$, and $a_i=\mathcal{H}$ if $u_i$ is half-activated in $G_i$ for each  $i\in\{1,...,n\}$.

\end{defn}

\begin{thm}
\label{main}
Let $G_1$ and $G_2$ be disjoint graphs with $u\in V(G_1)$, $w\in V(G_2)$, and $H=G_1uwG_2$. Let $\mathbf{s}$ be a solving pattern for the all-ones configuration on $H$, and $\Delta\nu:= \nu(H)-\nu(G_1)-\nu(G_2)$. Then we have the following table:
\begin{table}[h]
\label{table1}
\caption{Join Graph $H=G_1uwG_2$} 
\centering 
\begin{tabular}{|c|c|c|c|c|c|c|c|} 
\hline\hline 
\small  $u$ in $G_1$ & \small $w$ in $G_2$ & \small $u$ in $H$ & \small $w$ in $H$ & \small $\Delta\nu$ & \small $N(G_1)\mathbf{s_1}$ & \small $N(G_2)\mathbf{s_2}$ & \small {when}
\\ [0.5ex]
\hline 
$\mathcal{N}$&$\mathcal{N}$&$\mathcal{N}$&$\mathcal{N}$& $0$ & $\mathbf{1}$ & $\mathbf{1}$ &    \\[0ex]
\hline
$\mathcal{N}$&$\mathcal{A}$&$\mathcal{N}$&$\mathcal{A}$& $0$ & $\overline{\mathbf{c}_u}$ & $\mathbf{1}$ &     \\[0ex]
\hline
\multirow{2}{*}{$\mathcal{N}$} &\multirow{2}{*}{$\mathcal{H}$} & \multirow{2}{*}{$\mathcal{N}$} & \multirow{2}{*}{$\mathcal{H}$} & \multirow{2}{*}{$0$}  &  $\mathbf{1}$ & $\mathbf{1}$ & \small  $\mathbf{s}(w)=0$   \\
 \cline{6-8}
 & &  &  & &  $\overline{\mathbf{c}_u}$ & $\mathbf{1}$&  \small  $\mathbf{s}(w)=1$  \\
\hline
\multirow{2}{*}{$\mathcal{A}$} &\multirow{2}{*}{$\mathcal{A}$} & \multirow{2}{*}{$\mathcal{H}$} & \multirow{2}{*}{$\mathcal{H}$} & \multirow{2}{*}{$1$}  &  $\overline{\mathbf{c}_u}$ & $\mathbf{1}$ & \small  $\mathbf{s}(u)=0$, $\mathbf{s}(w)=1$  \\
 \cline{6-8}
 & &   &   &  & $\mathbf{1}$ & $\overline{\mathbf{c}_w}$ &  \small  $\mathbf{s}(u)=1$, $\mathbf{s}(w)=0$  \\[0ex]
\hline
$\mathcal{A}$ &$\mathcal{H}$&$\mathcal{N}$&$\mathcal{A}$& $-1$ & $\overline{\mathbf{c}_u}$ & $\mathbf{1}$ &    \\[0ex]
\hline
$\mathcal{H}$&$\mathcal{H}$&$\mathcal{N}$&$\mathcal{N}$& $-2$ & $\mathbf{1}$  & $\mathbf{1}$ &    \\[0ex]
\hline
\end{tabular}
$\mathcal{N}$: never-activated, $\mathcal{A}$: always-activated, $\mathcal{H}$: half-activated $\;\;\;\;\;\;\;\;\;\;\;\;\;$
\label{table1}
\end{table}
\\
where the not written cases can be obtained by symmetry.

\end{thm}
\begin{proof}
 Let $\mathbf{s}$ be a solving pattern for the all-ones configuration on $H$. Note that by Lemma \ref{cutedge} there are three cases we need to consider:

$\text{\emph{Case a}}:\; \mathbf{s}(u)=0,\; \mathbf{s}(w)=0$

$\text{\emph{Case b}}:\; \mathbf{s}(u)=0,\; \mathbf{s}(w)=1$

$\text{\emph{Case c}}:\; \mathbf{s}(u)=1,\; \mathbf{s}(w)=0.$

We will investigate each case under a specific Type-$(a_1,a_2)$ connection of $G_1$ and $G_2$. By symmetry we need to consider total of six types of connection:

Type-$(\mathcal{N},\mathcal{N})$ where $u$ is never-activated in $G_1$ and $w$ is never-activated in $G_2$: \emph{Case b} and \emph{Case c} are not possible. Indeed, for example, if  $\mathbf{s}(u)=1, \mathbf{s}(w)=0$ then $N(G_1)\mathbf{s_1}=\mathbf{1}$ with $\mathbf{s_1}(u)=1$, which contradicts with the fact that $u$ is a never-activated vertex of $G_1$. \emph{Case a} holds if and only if $N(G_1)\mathbf{s_1}=\mathbf{1}$, $N(G_2)\mathbf{s_2}=\mathbf{1}$. So the number of solutions is $2^{\nu(H)}=2^{\nu(G_1)}2^{\nu(G_2)}$ and  $u$ and $w$ are never-activated in $H$.

 Type-$(\mathcal{N},\mathcal{A})$ where $u$ is never-activated in $G_1$ and $w$ is always-activated in $G_2$: \emph{Case a} is not possible. Otherwise, $N(G_2)\mathbf{s_2}=\mathbf{1}$ with $\mathbf{s_2}(w)=0$, which contradicts with always-activatedness of $w$. Case $c$ is not possible because of the same reason in Type-$(\mathcal{N},\mathcal{N})$ connection. \emph{Case b} holds if and only if $N(G_1)\mathbf{s_1}=\overline{\mathbf{c}_u}$, $N(G_2)\mathbf{s_2}=\mathbf{1}$. So the number of solutions is $2^{\nu(H)}=2^{\nu(G_1)}2^{\nu(G_2)}$,  $u$ is never-activated in $H$ and $w$ is always-activated in $H$.

 Type-$(\mathcal{N},\mathcal{H})$ where $u$ is never-activated in $G_1$ and $w$ is half-activated in $G_2$: \emph{Case c} is not possible. Otherwise, $N(G_2)\mathbf{s_2}=\overline{\mathbf{c}_w}$, which implies $\mathbf{c}_w=\overline{\mathbf{c}_w}+\mathbf{1}$ is solvable, which contradicts with the half-activatedness of $w$ [\cite{Ballard19}, Proposition 2.4]. \emph{Case a} holds if and only if $N(G_1)\mathbf{s_1}=\mathbf{1}$, $N(G_2)\mathbf{s_2}=\mathbf{1}$ with $\mathbf{s_2}(w)=0$. Since $w$ is not activated in exactly half of the solutions for any configuration on $G_2$,  there are $2^{\nu(G_1)}2^{\nu(G_2)-1}$ solutions of the all-ones configuration on $H$ in \emph{Case a} . And \emph{Case b} holds if and only if $N(G_1)\mathbf{s_1}=\overline{\mathbf{c}_u}$, $N(G_2)\mathbf{s_2}=1$ with $\mathbf{s_2}(w)=1$ which corresponds $2^{\nu(G_1)}2^{\nu(G_2)-1}$ solutions of the all-ones configuration on $H$. In total, we have $2^{\nu(H)}=2^{\nu(G_1)}2^{\nu(G_2)}$ solutions in which $u$ is never-activated and $w$ is half-activated.

Type-$(\mathcal{A},\mathcal{A})$ where $u$ is always-activated in $G_1$ and $w$ is always-activated in $G_2$: \emph{Case a} is not possible. Otherwise, $N(G_1)\mathbf{s_1}=\mathbf{1}$ with $\mathbf{s_1}(u)=0$, which contradicts with always-activatedness of $u$. \emph{Case b} holds if and only if $N(G_1)\mathbf{s_1}=\overline{\mathbf{c}_u}$ and $N(G_2)\mathbf{s_2}=\mathbf{1}$. Hence, there are $2^{\nu(G_1)}2^{\nu(G_2)}$ solutions $\mathbf{s}$ on $H$, which satisfies \emph{Case b}. \emph{Case c} holds if and only if $N(G_1)\mathbf{s_1}=\mathbf{1}$, $N(G_2)\mathbf{s_2}=\overline{\mathbf{c}_w}$, which again corresponds $2^{\nu(G_1)}2^{\nu(G_2)}$ solutions $\mathbf{s}$ on $H$. In total, there are $2^{\nu(H)}=2^{\nu(G_1)+\nu(G_2)+1}$ solutions. Moreover, each vertex is activated only in one of the cases. Hence, $u$ and $w$ are half-activated in $H$

Type-$(\mathcal{A},\mathcal{H})$ where $u$ is always-activated in $G_1$ and $w$ is half-activated in $G_2$: \emph{Case a} and \emph{Case c} are not possible because of the same reason in Type-$(\mathcal{A},\mathcal{A})$ and Type-$(\mathcal{N},\mathcal{H})$ connections, respectively. \emph{Case b} holds if and only if $N(G_1)\mathbf{s_1}=\overline{\mathbf{c}_u}$ and $N(G_2)\mathbf{s_2}=\mathbf{1}$ with  $\mathbf{s_2}(w)=1$. Since $w$ is not activated in exactly half of the solutions for any configuration on $G_2$,  there are $2^{\nu(H)}=2^{\nu(G_1)}2^{\nu(G_2)-1}$ solutions of the all-ones configuration on $H$. Moreover, $u$ is never-activated and $w$ is always-activated in $H$.

Type-$(\mathcal{H},\mathcal{H})$ where $u$ is half-activated in $G_1$ and $w$ is half-activated in $G_2$:: \emph{Case c} is not possible because of the same reason in Type-$(\mathcal{N},\mathcal{H})$ connection. \emph{Case b} corresponds to the symmetric situation of \emph{Case c},  where $u$ and $w$ are switched. So it is not possible either. \emph{Case a} holds if and only if $N(G_1)\mathbf{s_1}=\mathbf{1}$ with $\mathbf{s_1}(u)=0$. Since  $u$ is half-activated there are $2^{\nu(G_1)-1}$  such solutions. Similarly,  $N(G_2)\mathbf{s_2}=\mathbf{1}$ with $\mathbf{s_2}(w)=0$ with $2^{\nu(G_2)-1}$ such solutions. In total, there are $2^{\nu(H)}=2^{\nu(G_1)+\nu(G_2)-2}$ solutions of the all-ones configuration on $H$. Moreover  $u$ and $w$ are never-activated in $H$.
\end{proof}

\begin{rem}
Let us assign a number $A(v)=A_G(v)$ to every vertex $v\in V(G)$, which we call activation number of $v$ in $G$ as follows. We say $A(v)=1$ if $v$ is always-activated,  $A(v)=0$ if $v$ is never-activated and $A(v)=-1$ if $v$ is half-activated. Then some of the information in the above theorem can be expressed more compactly as

\begin{align}
\Delta\nu= \left\{ \begin{array}{cc}
                -2 &  \;\;\text{if}\;\;\;\; A_{G_1}(u)=A_{G_2}(w)=-1\\
                A_{G_1}(u)A_{G_2}(w) &  \text{otherwise}  \\
                \end{array} \right\},
\end{align}

\begin{equation}
\label{Aa}
A_{H}(u)= A_{G_1}(u)(1+A_{G_2}(w))\mod 3,
\end{equation}

\begin{equation}
\label{Ab}
A_{H}(w)= A_{G_2}(w)(1+A_{G_1}(u))\mod 3.
\end{equation}
\end{rem}

We want to note that the fifth column of Table \ref{table1} is also deducible from [\cite{Ballard19}, Theorem 2.18] using the relation between activation numbers and null differences. We have the following corrolary.

\begin{cor}
Let $F$,$G_1$,...,$G_n$ be always solvable graphs. Let $u\in V(F)$, and $v_i \in V(G_i)$ for $1\leq i\leq n$. Consider the join graph $H$ obtained by joining $v_i$ to $u$ for $1\leq i\leq n$ by an edge. If $u$ is never-activated in $F$, then $H$ is always solvable. If $u$ is always-activated in $F$, then $H$ is always solvable if and only if $v_i$ is always-activated in $G_i$ for even number of vertices.
\end{cor}
\begin{proof}
Note that $H$ can be obtained by applying Type-$(\cdot,\cdot)$ connections successively to the graphs  $F$,$G_1$,...,$G_n$, where at each connection one of the vertices is always taken as $u$. The first part of the corollary follows from the fact that Type-$(\mathcal{N},\cdot)$ connections never change the activation types of the vertices nor $\Delta\nu$. On the other hand, single Type-$(\mathcal{A},\mathcal{A})$ connection increases $\Delta\nu$ by $1$. However, joining even number of $v_i$ to $u$ corresponds a series of Type-$(\mathcal{A},\mathcal{A})$ connection followed by a Type-$(\mathcal{H},\mathcal{A})$ connection, which remains $\Delta\nu$ unchanged.

\end{proof}

\section{Always solvable trees}

In this section, we use the following observation: Since every edge $e$ of a tree $T$ is a cut edge, $T$ can always be seen as a join graph of the connected components of $T-e$. As a result, every adjacent pair of vertices $u$ and $w$  can only have the pair of activation types
determined by the third and fourth column of Table \ref{table1}. And the possible values of their activation types with respect to the connected components of $T-(u,w)$ can be one of those corresponding pairs in the first and second column of Table \ref{table1}.

From [\cite{Amin92}, Lemma 2], we know that the number of half-activated vertices of a graph is even. Moreover, proof of [\cite{Amin92}, Lemma 2] also implies the following.
\begin{lem}
\label{adj}
Let $u$ be a half-activated vertex of a graph. Then there exists a half-activated vertex adjacent to $u$.
\end{lem}
\begin{proof}
By definition, $u$ is a half-activated vertex of a graph $G$ if  $\boldsymbol{\ell}(u)=1$ for some null pattern $\boldsymbol{\ell}$. Let us denote the characteristic vector of the closed neighborhood set of $u$ by $\mathbf{N}[u]$. Since $\boldsymbol{\ell}$ is a null pattern $\mathbf{N}[u]\cdot\boldsymbol{\ell}=(N(G)\boldsymbol{\ell})(u)=\mathbf{0}(u)=0$. Since $\boldsymbol{\ell}(u)=1$, this implies $\boldsymbol{\ell}$ activates odd number of vertices adjacent to $u$. So there exists at least one vertex $w$ adjacent to $u$ such that  $\boldsymbol{\ell}(w)=1$. Hence, $w$ is half-activated.
\end{proof}
\begin{defn}
Let $G$ be a nonempty graph and $P=\{G_1,...,G_k\}$ be a set of subgraphs of $G$.  We say $G$ is partitioned into $P$ (or $P$ partitions $G$) if $V(G_i)$'s are pairwise disjoint and $V(G)=\bigcup_{i=1}^k V(G_i)$. If all $G_i$'s are connected and always solvable, we say $P$ is a partition into always solvable subgraphs (\emph{PASS}) of $G$.  We say a \emph{PASS} $M$ is minimal if for all \emph{PASS}'s $P$ of $G$, $|M|\leq |P|.$  We define the number $\pi(G)$ as the cardinality of a minimal \emph{PASS} of $G$.
\end{defn}
\begin{rem}
Note that since $K_1$ is always solvable, every graph $G$ has an at least one PASS, which consists of single vertex subgraphs of $G$. Consequently, every graph has a minimal PASS and each minimal PASS has the same cardinality. Hence, $\pi(G)$ is well defined for all graphs $G$.
\end{rem}
\begin{thm}
\label{partition}
For a nonempty tree $T$, $\pi(T)=\nu(T)+1$. In other words, $T$ has nullity $n$ if and only if there exists a minimal PASS of $T$ with cardinality $n+1$.
\end{thm}
\begin{proof}
We first show $\pi(T)\leq\nu(T)+1$, which is equivalent to prove that there exists a \emph{PASS} of $T$ with cardinality $\nu(T)+1$. We prove the claim by applying induction on the nullity of the tree. Taking the partition as $T$ itself, we see that the claim trivially holds true for $\nu(T)=0$. Assume that it holds true for all
 trees $S$ with $\nu(S)<n$. If $T$ is a tree with $\nu(T)=n$, then it has a half-activated vertex. Moreover, by Lemma \ref{adj} there exists an adjacent pair of half-activated vertices $x$ and $y$ of $T$. Let $X$ and $Y$ be the components of $T-(x,y)$ containing $x$, and $y$, respectively.
From Table \ref{table1} we see $T$ must be a Type-$(1,1)$ connection of $X$ and $Y$, and $\nu(X)+\nu(Y)=\nu(T)-\Delta\nu=n-1$. So both $\nu(X)$ and $\nu(Y)$ are less than $n$. By induction hypothesis there exist \emph{PASS}'s $\{X_1,...,X_r\}$ and $\{Y_1,...,Y_s\}$ of $X$ and $Y$, respectively with $r=\nu(X)+1$ and $s=\nu(Y)+1$. Note that since $X$ and $Y$ are disjoint, $X_i$ and $Y_j$ are disjoint as well for all $1\leq i\leq r$, $1\leq j\leq s$. Moreover, $V(T)=V(X)\cup V(Y)$. Thus $\{X_1,...,X_r,Y_1,...,Y_s\}$ is a \emph{PASS} of $T$ with cardinality $r+s=\nu(X)+1+\nu(Y)+1=n-1+2=n+1$. This proves $\pi(T)\leq\nu(T)+1$.

To prove the converse inequality let $M=\{T_1,...,T_m\}$ be a minimal \emph{PASS} of $T$ with $m=\pi(T)$.
Since $T$ is a tree, one can easily observe that there cannot be more than one edge between different subtrees $T_i$ and $T_j$ in $T$. Moreover, each subtree $T_i$ is connected by an edge to at least one other subtree $T_j$; and every edge between subtrees is a cut edge. These observations allow us to realize that the number of edges between subtrees $T_i$'s equals to $m-1$, and $T$ can be seen as a join graph obtained by applying $m-1$ successive Type-$(\cdot,\cdot)$ connections on $T_i$'s. Since initially all subgraphs has nullity $0$, and since the nullity of any join graph obtained by a single Type-$(\cdot,\cdot)$ connection exceeds the sum of the nullities of joined components at most by $1$, we conclude that $\nu(T)$ can at most be $m-1=\pi(T)-1$. Thus, $\nu(T)+1 \leq \pi(T) $.
\end{proof}

\begin{rem}
Note that the above identity cannot be generalized to arbitrary graphs. Indeed, if we consider the cycle $C_6$, we see that it has a minimal PASS, which consists of two subgraphs isomorphic to $P_3$. Thus $\pi(C_6)=2$. However, $\nu(C_6)=2$ as well.
\end{rem}

Let $S_m$ denote the star tree which is the tree with order $m$ with maximum diameter $2$. In other words, $S_1$ is the single vertex graph $K_1$ and $S_m$ is the complete bipartite graph $K_{1,m-1}$ for $m> 1$. We have the following main result.
\begin{thm}
\label{chr2}
Let $T$ be a tree different than $S_m$, $m\in \mathbb{N}$. Then, $T$ is always solvable if and only if $T$ is a Type-$(\mathcal{N},\mathcal{A})$ or Type-$(\mathcal{N},\mathcal{N})$ connection of two always solvable trees.
\end{thm}
\begin{proof}
($\Rightarrow$) Assume for a contradiction that $T$ is an always solvable tree different than $S_m$ but it is not a Type-$(\mathcal{N},\mathcal{A})$ or Type-$(\mathcal{N},\mathcal{N})$ connection of always solvable trees. Since $T$ is always solvable it does not have any half-activated vertex. So every vertex is either always or never-activated. Since $T$ is not $S_1$ and two always activated vertex cannot be adjacent in a tree \cite{Conlon99}, $T$ has a never activated vertex $u$.  Let $\{v_1,...,v_n\}$ be an enumeration of neighbors of $u$. Furthermore, for each $i\in \{1,...,n\}$, let $U_i$ and $T_i$ be the connected components of $T-(u,v_i)$ containing $u$ and $v_i$, respectively.

First assume that $v_i$ is always-activated in $T$. Since $T$ is not a Type-$(\mathcal{N},\mathcal{A})$ connection of always solvable trees, we see from Table \ref{table1} that $T$ is a Type-$(\mathcal{A},\mathcal{H})$ connection of $U_i$ and $T_i$, respectively. Moreover, $v_i$ is half activated in $T_i$. So $\nu(T_i) \geq 1$. On the other hand, $\nu(U_i)+\nu(T_i)=\nu(T)-\Delta \nu=0-(-1)=1.$ Hence $\nu(T_i)=1$.

Second assume that $v_i$ is never-activated in $T$. Since $T$ is not a Type-$(\mathcal{N},\mathcal{N})$ connection of always solvable trees, we see from Table \ref{table1} that $T$ is a Type-$(\mathcal{H},\mathcal{H})$ connection of $U_i$ and $T_i$, respectively. Moreover, $u$ and $v_i$ are half activated in $U_i$ and $T_i$, respectively. So $\nu(U_i) \geq 1$ and $\nu(T_i) \geq 1$. On the other hand, $\nu(U_i)+\nu(T_i)=\nu(T)-\Delta \nu=0-(-2)=2.$ Hence $\nu(U_i)=\nu(T_i)=1$.

So in both cases we see that $\nu(T_i)=1$ and $v_i$ is half-activated in $T_i$.

Now define the trees $R_k$ recursively as follows. Let $R_1= K_1uv_1T_1$ and $R_k=R_{k-1}uv_kT_k$ for $k\geq 2$. Since $u$ is always activated in the single vertex graph $K_1$ and $v_1$ is half-activated in $T_1$, $R_1$ is a Type-$(\mathcal{A},\mathcal{H})$ connection of $K_1$ and $T_1$, respectively. Hence, $\nu(R_1)=\nu(K_1)+\nu(T_1)+\Delta\nu=0+1-1=0$. Moreover $u$ is never-activated in $R_1$. Then $R_2$ is a Type-$(\mathcal{N},\mathcal{H})$ connection of $R_1$ and $T_2$. Hence $\nu(R_2)=\nu(R_1)+\nu(T_2)+\Delta\nu=0+1-0=1$ and $u$ is never-activated in $R_2$. Continuing this way we see that for all $k \geq 2$, $R_k$ is a Type-$(\mathcal{N},\mathcal{H})$ connection of $R_{k-1}$ and $T_k$ and $\nu(R_{k})=k-1$. Note that $R_n=T$. So $n-1=\nu(R_n)=\nu(T)=0,$ which gives $n=1$.

This proves that every never-activated vertex of $T$ is a leaf. Together with the fact that two always activated vertex cannot be adjacent in a tree, this implies $T$ is the star tree $S_m$ for some $m>1$, which contradicts with our assumption.

($\Leftarrow$) From Table \ref{table1} we see that in both types of connection $\Delta\nu=0$. Hence the resulting tree is always solvable.
\end{proof}

\section{Trees with adjacent pair of never activated vertices}

In this section we investigate the trees with adjacent pair of never activated vertices.

\begin{thm}
\label{thmP}
Let $T$ be a tree with exactly two even degree vertices and $P$ be the path between these vertices with the vertex sequence $(v_1,...,v_n)$. Then, if the length of $P$ is odd, we have $\nu(T)=0$ and $v_i$ is never-activated
for all $i$. If the length of $P$ is even, then we have $\nu(T)=1$ and $v_i$ is never-activated if $i$ is even, half-activated if $i$ is odd.

\end{thm}

\begin{proof}

Let us first assume that the length of $P$ is $1$. Let $e=(v_1,v_2)$ and $T_1$, $T_2$ be the connected components of $T-e$ containing $v_1$ and $v_2$, respectively. Then $T_1$ and $T_2$ are odd trees (Every vertex of $T_1$ and $T_2$ has odd degree). By [\cite{Amin98}, Theorem 3] $\nu(T_1)=\nu(T_2)=1.$ Moreover, $\mathbf{1}$ is a null pattern of both $T_1$ and $T_2$ (Here  we use the same notation $\mathbf{1}$ to denote the all-ones patterns of different graphs). This implies $v_1$ and $v_2$ are half-activated in $T_1$ and $T_2$, respectively.  Note that $T=T_1v_1v_2T_2$. Hence, by Table \ref{table1}, $\nu(T)=\nu(T_1)+\nu(T_2)+\Delta \nu=1+1-2=0$ and $v_1$ and $v_2$ are never-activated in $T$.

Second assume that the length of $P$ is $2$. Then $P$ has the vertex sequence $(v_1,v_2,v_3)$. Let  $e=(v_1,v_2)$ and $T_1$, $T_2$ be the connected components of $T-e$ containing $v_1$ and $v_2$, respectively. Then $T_1$ is an odd tree. Hence, $\nu(T_1)=1$ by [\cite{Amin98}, Theorem 3], and $v_1$ is half-activated in $T_1$. On the other hand, $T_2$ is a tree with exactly two even degree vertices $v_2$ and $v_3$ and the path between them has length $1$. By the previous discussion $\nu(T_2)=0$ and $v_2$ is never-activated in $T_2$. Since  $T=T_1v_1v_2T_2$, by Table \ref{table1} we have $\nu(T)=\nu(T_1)+\nu(T_2)+\Delta \nu=1+0+0=1$. Moreover, $v_1$ is half-activated and $v_2$ is never-activated in $T$. By symmetry we also see that $v_3$ is half-activated in $T$.

Now assume  that the length of $P$ is $k>2$ and the claim holds true for all trees with exactly two even degree vertices with the path between them has length less than $k$.  Take a vertex $v_i$ from the vertex sequence of $P$ such that $1\leq i<k+1$. Let $e=(v_i,v_{i+1})$ and $T_1$, $T_2$ be the connected components of $T-e$ containing $v_i$ and $v_{i+1}$, respectively.

 First assume that $k$ is odd. If $i$ is different then $1$ or $k$, then $T_1$ and $T_2$ are trees with exactly two vertices ($v_1$, $v_i$ in $T_1$ and $v_{i+1}$, $v_{k+1}$ in $T_2$ ) with even degree. Assume that $i$ is even, hence the paths between the even degree vertices of $T_1$ and $T_2$ have both odd lengths. Then by the induction hypothesis $\nu(T_1)=\nu(T_2)=0$. Moreover, $v_i$ and $v_{i+1}$ are never-activated in $T_1$ and $T_2$ respectively. Then, by Table \ref{table1} we have $\nu(T)=\nu(T_1)+\nu(T_2)+\Delta \nu=0+0+0=0$ and $v_i$ and $v_{i+1}$ are never-activated in $T$. This proves $v_i$ is never-activated in $T$ except for $i=1$ and $i=k+1$. For $i=1$, $T_1$ is an odd tree and $T_2$ is a tree with exactly two even degree vertices $v_{2}$, $v_{k+1}$ with the path between them has even length. By [\cite{Amin98}, Theorem 3] and by the induction hypothesis $\nu(T_1)=\nu(T_2)=1$ and $v_1$, $v_2$ are half-activated in $T_1$, $T_2$,  respectively. Then, by Table \ref{table1} we see that $v_1$ is never-activated in $T$. By symmetry,  $v_{k+1}$ is never-activated in $T$ as well.

Second assume that $k$ is even. If $i$ is even and different than $k$, then $T_1$ is a tree with exactly two even degree vertices $v_1$, $v_i$ with the path between them has odd length. And $T_2$ is a tree with exactly two even degree vertices $v_{i+1}$, $v_{k+1}$ with the path between them has even length. Then by induction hypothesis $\nu(T_1)=0$, $\nu(T_2)=1$, $v_i$ is never-activated in $T_1$, and $v_{i+1}$ is half-activated in $T_2$. By Table \ref{table1} we have $\nu(T)=\nu(T_1)+\nu(T_2)+\Delta \nu=1+0+0=1$, $v_i$ is never-activated and $v_{i+1}$ is half-activated in $T$. This proves that $v_i$ is never-activated in $T$ if $i$ is even and half-activated in $T$ if $i$ is odd except for $i=1, k,k+1$. We immediately see that $v_{k}$ is never-activated in $T$ using symmetry since $v_{2}$ is never-activated in $T$. And if $i=k$ then  $T_1$ is a tree with exactly two even degree vertices $v_1$, $v_{k}$ with the path between them has odd length, and $T_2$ is an odd tree. By [\cite{Amin98}, Theorem 3] and by the induction hypothesis $\nu(T_1)=0$, $\nu(T_2)=-1$, $v_{k}$ is never-activated in $T_1$, and $v_{k+1}$ is half-activated in $T_2$. Then, by Table \ref{table1} we see that $v_{k+1}$ is half-activated in $T$, and by symmetry $v_{1}$ is half-activated in $T$.
\end{proof}

\begin{thm}
\label{thmN}
 Let $T$ be a tree with number of even degree vertices is equal to $n\neq 0$. Then, $\nu(T) \leq n-1$. Moreover, if there exists a pair of even degree vertices where the path between them has odd length, then $\nu(T) \leq n-2$.
\end{thm}
\begin{proof}
Let $\{v_1,...,v_n\}$ be an enumeration of the set of even degree vertices of $T$. Let, $T_0=T$ and $T_1$ be the tree obtained from $T$ by joining $v_1$ to a single vertex $u_1$ by an edge. Let $T_2$ be the tree obtained from $T_1$ by joining $v_2$ to a single vertex $u_2$ by an edge. Continuing this way we construct the tree $T_i$, $i\in \{1,...,n-1\}$ where for each $k\leq i$, $v_k$ is joined to a single vertex $u_k$ by an edge. Since the single vertex of $K_1$ is always activated in $K_1$, and $K_1$ is always solvable, we see from Table \ref{table1} that $\nu(T_i)-\nu(T_{i-1}) \geq -1$. This implies $\nu(T_{n-1})-\nu(T) \geq -(n-1)$. On the other hand, observe that $T_{n-1}$ is a tree with a single even degree vertex $v_n$. Hence, by [\cite{Amin96}, Theorem 3], $\nu(T_{n-1})=0$. Together with previous inequality we obtain $\nu(T) \leq n-1$.

Now assume that there exists a pair of even degree vertices where the path between them has odd length. Take the enumeration of even degree vertices such that $v_{n-1}$ and $v_n$ are these two vertices. Then $\nu(T_{n-2})=0$ by Theorem \ref{thmP}. Moreover, $\nu(T_{n-2})-\nu(T) \geq -(n-2)$. Hence $\nu(T) \leq n-2$.

\end{proof}

\begin{lem}
\label{lemend}
Let $P$ be a maximal length path of not activated vertices in a solving pattern for the all-ones configuration in a tree. Then the end vertices of $P$ has even degree.
\end{lem}

\begin{proof}
Let $\mathbf{s}$ be a solving pattern for the all-ones configuration and $P$ be a maximal length path of not activated vertices of $\mathbf{s}$. Let $u$ be an end vertex of $P$, $v$ be the vertex in $P$ adjacent to $u$. Since $P$ is a maximal path of not activated vertices,  all vertices adjacent to $u$ other than $v$ must be activated in $\mathbf{s}$. Moreover since $\mathbf{s}$ is a solving pattern for the all ones configuration, and since $u$ and $v$ are not activated, number of vertices adjacent to $u$ other than $v$ must be odd. Hence, degree of $u$ is even.
\end{proof}

\begin{thm}
\label{thmact}
Let $T$ be a tree with at least two even degree vertices. Then, all even degree vertices of $T$ have only even length paths between them if and only if there exists a solving pattern for the all-ones configuration where each even degree vertex is activated.
\end{thm}

\begin{proof}
($\Leftarrow$) We use contraposition for this side of the proof. So assume that there exist even degree vertices $u$ and $w$ with odd length path $P$ between them. Then we want to prove for all solving patterns for the all-ones configuration there exists an even degree vertex which is not activated. Let $\mathbf{s}$ be a solving pattern for the all-ones configuration. Assume $u$ and $w$ are both activated in $\mathbf{s}$, otherwise we are done. Since the length of $P$ is odd and since two activated vertices can not be adjacent in a tree \cite{Conlon99}, there must exists an adjacent pair of not activated vertices of $P$. Hence, there exists a maximal path of not activated vertices.  Then, by Lemma \ref{lemend}, the result follows.

($\Rightarrow$) Now let us assume all even degree vertices of $T$ have only even length paths between them. Let $\sim$ be the equivalence relation on the vertex set such that $u\sim w$ if the path between $u$ and $w$ has even length. Hence, $\sim$ partitions $V(T)$ into two subsets $A$ and $B$, one of which contains all even degree vertices by assumption, let it be $A$. Then, it is easy to see that the characteristic vector of the set $A$ is a solving pattern for the all-ones configuration on which every even degree vertex is activated.

\end{proof}

\begin{thm}
Let $T$ be a tree. Then the followings are equivalent.

(1) $T$ is always solvable and it has even order vertices with the path between them has odd length.

(2) $T$ is always solvable and it has adjacent pair of never activated vertices.

(3) $T$ is obtained from always solvable trees by a Type-$(\mathcal{N},\mathcal{N})$ or a Type-$(\mathcal{A},\mathcal{A},\mathcal{A},\mathcal{A})$ connection.

\end{thm}

\begin{proof}
\emph{Proof of }$(1) \Rightarrow (2)$: Let $P$ be an odd length path between two even degree vertices in $T$, $\mathbf{s}$ be the unique solving pattern for the all-ones configuration.  First recall that two activated vertices in a solving pattern for the all-ones configuration can not be adjacent in a tree \cite{Conlon99}. Together with the assumption that $P$ has odd length, in order not to have an adjacent pair of not activated vertices of $P$ in $\mathbf{s}$, we must have one end vertex of $P$ is not activated while the other end is activated. However if an even degree vertex is not activated in $\mathbf{s}$, then there must have a not activated vertex adjacent to it in order it to have odd number of activated neighbors. The result follows by the fact that being not activated in $\mathbf{s}$ is equivalent to being never activated since $\mathbf{s}$ is the unique solving pattern for the all-ones configuration.

\emph{Proof of }$(2) \Rightarrow (1)$: Let $T$ be always solvable and have adjacent pair of never activated vertices. Then there exist never activated even degree vertices of $T$ by Lemma \ref{lemend}. So the result follows by Theorem \ref{thmact}.

\emph{Proof of }$(2) \Rightarrow (3)$:  Let $T$ be always solvable and have adjacent pair of never activated vertices $x$, $y$. $T=RxyS$ where $R$ and $S$ are the connected components of $T-(x,y)$ containing $x$, $y$, respectively. From Table \ref{table1},
we see that $T$ must be either Type-$(\mathcal{N},\mathcal{N})$ or Type-$(\mathcal{H},\mathcal{H})$ connection of $R$ and $S$, respectively. If the former holds $\nu(R)+\nu(S)=\nu(T)-\Delta\nu=0-0=0$, which implies $R$ and $S$ are always solvable.
So we are done. If the latter holds
$\nu(R)+\nu(S)=\nu(T)-\Delta\nu=0-(-2)=2$. Moreover, $x$ and $y$ are half-activated in $R$ and $S$, respectively. Hence $\nu(R)$ and $\nu(S)$ are nonzero. This implies $\nu(R)=\nu(S)=1$. By Lemma \ref{adj}, there exists a half-activated vertex $w$ of $R$ which is adjacent to $x$ and a half-activated vertex  $z$ of $S$ which is adjacent to $y$. Let $R=WwxX$ where $W$ and $X$ are the connected components of $R-(w,x)$ containing $w$, $x$, respectively. Similarly let $S=YyzZ$ where $Y$ and $Z$ are the connected components of $S-(y,z)$ containing $y$, $z$, respectively.  From
Table \ref{table1} we see $R$ and $S$ are  Type-$(\mathcal{A},\mathcal{A})$ connections of $W$, $X$ and  $Y$, $Z$, respectively. Consequently, $\nu(W)+\nu(X)=\nu(R)-\Delta\nu=1-1=0$, $\nu(Y)+\nu(Z)=\nu(S)-\Delta\nu=1-1=0$ which implies $W$, $X$, $Y$, and  $Z$ are always solvable. Moreover, $w$, $x$, $y$, and $z$ are always-activated in $W$, $X$, $Y$, and $Z$, respectively. Since $T=WwxXxyYyzZ$, this implies $T$ is a Type-$(\mathcal{A},\mathcal{A},\mathcal{A},\mathcal{A})$ connection of always solvable trees.

\emph{Proof of }$(3) \Rightarrow (2)$: From Table \ref{table1} it is clear that Type-$(\mathcal{N},\mathcal{N})$ connection does not change total nullity and produces a pair of never activated vertices. On the other hand a Type-$(\mathcal{A},\mathcal{A},\mathcal{A},\mathcal{A})$ connection of always solvable graphs is a Type-$(\mathcal{H},\mathcal{H})$ connection of two graphs $R_1$ and $R_2$ with nullity $1$ which are themselves produced from Type-$(\mathcal{A},\mathcal{A})$ connection of always solvable graphs. Again from Table \ref{table1} it is clear to see that Type-$(\mathcal{H},\mathcal{H})$ connection produces a pair of never activated vertices. Moreover, the resulting tree has nullity $\nu(R_1)+\nu(R_2)-\Delta\nu=1+1-2=0$.

\end{proof}

\bibliographystyle{plain}

\end{document}